\documentclass[12pt]{amsart}

\usepackage{amsthm,amssymb,amstext,amscd,amsfonts,amsbsy,amsxtra,latexsym,amsmath}
\usepackage{a4wide}
\usepackage[english]{babel}
\usepackage[latin1]{inputenc}
\usepackage{verbatim}
\usepackage{graphicx}  
\usepackage[all]{xy} 
\numberwithin{figure}{section}

\textheight7.5in \textwidth6.05in \hoffset 0.21in\newcommand{\field}[1]{\mathbb{#1}}
\newcommand{\N}{\field{N}}
\newcommand{\Z}{\field{Z}}

\newcommand{\C}{\field{C}}
\newcommand{\Q}{\field{Q}}

\newcommand{\im}{\text{Im}}
\newcommand{\re}{\text{Re}}

\newcommand{\CM}{\mathcal{M}}

\newcommand{\non}{\nonumber}

\numberwithin{equation}{section}
\newtheorem{theorem}{\textbf{Theorem}}
\numberwithin{theorem}{section}
\newtheorem{corollary}[theorem]{\textbf{Corollary}}
\newtheorem{lemma}[theorem]{\textbf{Lemma}}
\newtheorem{proposition}[theorem]{\textbf{Proposition}}

\newtheorem*{remark}{Remark}

\renewenvironment{proof}[1][Proof]{\begin{trivlist}
\item[\hskip \labelsep {\bfseries #1:}]}{\qed\end{trivlist}}

\newcommand{\bea}{\begin{eqnarray}} 
\newcommand{\eea}{\end{eqnarray}} 
\newcommand{\be}{\begin{equation}} 
\newcommand{\ee}{\end{equation}} 


\begin{document}

\title[]{Asymptotic formulas for coefficients of inverse theta functions}

\author{Kathrin Bringmann$^1$}
\address{$^1$ Mathematical Institute\\University of
Cologne\\ Weyertal 86-90 \\ 50931 Cologne \\Germany}
\email{kbringma@math.uni-koeln.de}

\author{Jan Manschot$^{2}$}
\address{$^{2}$ Camille Jordan Institute \\ University of Lyon\\ 43 boulevard du 11 novembre 1918\\
69622 Villeurbanne cedex\\ 
France}
\email{jan.manschot@univ-lyon1.fr}

\begin{abstract} 
We determine asymptotic formulas for the coefficients of a natural class of
negative index and negative weight Jacobi forms. These coefficients
can be viewed as a refinement of the numbers $p_k(n)$ of partitions of $n$ into $k$
colors. Part of the motivation for this work is that they are equal to
the Betti numbers of the Hilbert scheme of points on an algebraic
surface $S$ and appear also  as counts of Bogomolny-Prasad-Sommerfield
(BPS) states in physics.  
\end{abstract}

\maketitle

\section{Introduction and Statement of Results}
Jacobi forms were first systematically studied by Eichler and Zagier \cite{eichler} and enjoy a wide variety of applications in
the theory of modular forms, combinatorics \cite{Wright:1968,Wright:1971}, 
conformal field theory \cite{DiFrancesco:1997nk, Kawai:1993jk},
black hole physics \cite{Dabholkar:2012nd, Dijkgraaf:1996xw}, Hilbert schemes of points \cite{Gottsche:1990}, Donaldson invariants \cite{Gottsche:1996}, and many other topics. This paper focuses on a
class of negative index Jacobi forms with a 
single order pole in the elliptic variable $w$. The analysis of the coefficients of such functions
is more complicated then the well-understood class of Jacobi forms which depend holomorphically on $w$. It turns out that these Fourier coefficients (in $w$) are not modular but related to quantum modular forms \cite{FOR, Z2}. The appearance of these  
functions in the above mentioned topics calls for an explicit knowledge of their coefficients and in
particular of their asymptotic growth. In this paper we provide
such asymptotic formulas. One of the immediate motivations is the
counting of BPS states in physics and in particular those with vanishing angular
momentum. This motivation is explained in more detail after stating
the results.
 
\subsection{Statement of results}
We consider the following class of negative weight $1-k/2$ ($k\in \N$) and index $-1/
2$ Jacobi forms 
$$
h_k\left(w;\tau\right):=\frac{i}{\theta_1(w;\tau)\eta(\tau)^{k-3}}, 
$$
with $(q:=\exp(2\pi i \tau)$, $\zeta:=\exp(2\pi i w))$
\begin{eqnarray}
&&\theta_1(w; \tau):=i \zeta^{\frac{1}{2}} q^{1\over 8}  \prod_{n\geq 1}(1-q^n)(1-\zeta q^n)\left(1-\zeta^{-1}q^{n-1}\right),\non \\
&&\eta(\tau):=q^{\frac{1}{24}}\prod_{n\geq 1}(1-q^n).\non
\end{eqnarray}
We are interested in two expansions of $h_k(w;\tau)$. The first
expansion is in terms of the coefficients $a_{m,k}(n)$, defined by:
\begin{eqnarray}
\label{eq:genfunc}
 q^{\frac{k}{24} }\left(\zeta^{1\over 2}-\zeta^{-{1\over 
        2}}\right) h_{k}(w;\tau)=:\sum_{n\geq 0 \atop m\in
  \mathbb{Z}} a_{m,k}(n) \zeta^m q^n, \qquad \quad |\zeta q|,\, |\zeta^{-1} q|<1.
\end{eqnarray}
Note that $\sum_{m\in\mathbb{Z}} a_{m,k}(n)=p_{k}(n)$, where $p_{k}(n)$ denotes the  number
of partitions of $n$ into $k$ colors. These are enumerated by $\eta(\tau)^{-k}$: 
\[
\sum_{n\geq 0} p_{k}(n)\,q^{n-{k\over 24}}=\frac{1}{\eta(\tau)^k}.
\] 
For $k=1$, equation (\ref{eq:genfunc}) corresponds to the generating
function of the crank of partition \cite{Garvan:1988}, and for $k=2$ to the
birank \cite{Hammond:2004}.

The second expansion is motivated from physics, and is based on the fact that the coefficients of $q$
are Laurent polynomials, symmetric under $\zeta\to \zeta^{-1}$ and with maximal degree
$n$. Therefore, we can express $h_k(w;\tau)$ as:
\be
\label{eq:genfunc2}
q^{k\over 24 }\left(\zeta^{1\over 2}-\zeta^{-{1\over
        2}}\right) h_{k}(w;\tau)=\sum_{m,n\geq 0} b_{m,k}(n)\,  \chi_{2m+1}\!\left(\zeta^{1\over 2}\right)\, q^n,
\ee
with
\[
\chi_m(\zeta):=\frac{\zeta^{m}-\zeta^{-m}}{\zeta-\zeta^{-1}},\qquad \qquad b_{m,k}(n):=a_{m,k}(n)-a_{m+1,k}(n).
\]

Following the approach of Wright \cite{Wright:1971}, we determine all
polynomial corrections to the leading exponential of the coefficients 
$a_{m,k}(n)$ in the large $n$ limit. 
\begin{theorem} 
\label{mainasymp}
We have for $N\geq 1$:
\begin{align*}
a_{m, k}(n)&=(2\pi)^{-\frac{k}{2}} \sum_{\ell=1}^N d_{m,k}(\ell)
n^{-\frac{2+2\ell+k}{4}}
\left(\pi \sqrt{\frac{k}{6} } \right)^{1+\ell+\frac{k}{2}}\\
&\quad \times I_{-1-\ell-\frac{k}{2}}\left(\pi \sqrt{\frac{2kn}{3}}\right)+O\left(n^{-1-\frac{N}{2}-\frac{k}{4}} e^{\pi\sqrt{\frac{2kn}{3}}}\right),
\end{align*}
where $d_{m,k}(\ell)$ are defined by equation (\ref{eq:ZagierTaylor})
 and $I_\ell(x)$ 
is the usual $I$-Bessel function.  Here the error term depends on $k$ and $m$.
\end{theorem} 
Theorem \ref{mainasymp} allows us to compare the asymptotic growths of $a_{m,k}(n)$ for different values of $m$. 
The asymptotic behavior of the Bessel function :
\[
I_\ell(x)=\frac{e^x}{\sqrt{2\pi x}}\left(1+O\left(x^{-1}\right) \right),
\]
directly yields:
\begin{corollary} 
We have
\[
a_{m, k}(n)-a_{r, k}(n)= \pi^3\left(r^2-m^2\right)
(8n)^{-\frac{9+k}{4}}\left(\frac{k}{3}
\right)^{\frac{k+7}{4}}e^{\pi\sqrt{\frac{2kn}{3}}}+O\left(n^{-3-\frac{k}4} e^{\pi \sqrt{\frac{2kn}3}} \right),
\] 
where the error term depends on $m$, $k$, and $r$.
\end{corollary}  
\noindent 

From  Theorem \ref{mainasymp} the asymptotics of $b_{m,k}(n)$ for large $n$ immediately follow: 
\begin{corollary} We have 
\be
\label{eq:asympalpha}
b_{m,k}(n)=(2m+1) \,\pi^3\,\,(8n)^{-{9+k \over 4}} \left(\frac{k}{3}
\right)^{ k+7 \over 4} e^{\pi \sqrt{2kn \over 3}} +O\left(n^{-3-{k\over
    4}}e^{\pi\sqrt{2kn \over 3}}\right),
\ee
where the error term depends on $m$ and $k$.
\end{corollary}
\noindent Note that this corollary shows that $b_{m,k}(n)$ increases with $m$ in
the limit of large $n$. Beyond the validity of equation (\ref{eq:asympalpha}), $b_{m,k}(n)$ eventually
decreases with increasing $m$ for fixed $n$, and in particular $b_{m,k}(n)=0$ for $m>n$.

 We next compare the asymptotic behavior of the coefficients $b_{0,k}(n)$ with those of
$p_{k}(n)$. It is well-known that the asymptotic growth of the latter is given by
\cite{HR2, rademacher1938:1}:

\[
p_{k}(n)=2 \left(\frac{k}{3} \right)^{1+k \over 4} (8n)^{-{3+k \over
    4}}e^{\pi\sqrt{2kn \over 3}} + O\left(n^{-{5+ k \over 4}}e^{\pi\sqrt{2kn \over 3}}\right).
\]
Thus we find for the ratio $b_{0,k}(n)/p_{k}(n)$:
\[
\frac{b_{0,k}(n)}{p_{k}(n)}=\frac{\pi^3}{16}\left(
  \frac{k}{3n}\right)^{3\over 2}+O\left(n^{-2}\right).
\]

\subsection{Motivation:  Moduli spaces and BPS states}

BPS states of both gauge theory and gravity have been extensively studied in the past for a variety of
reasons. These states are representations of the $\operatorname{SU}(2)_{\rm spin}$ massive little
group in four dimensions labeled by their angular momentum or highest
weight $J$. The subset
of BPS states with vanishing angular momentum ($J=0$), also known
as ``pure Higgs states'' \cite{Bena:2012hf}, have recently attracted much interest 
\cite{Bena:2012hf, Lee:2012sc, Manschot:2012rx, Sen:2009vz}. The
states with $J=0$ are in some sense more fundamental. In particular in
gravity, these states are candidates for microstates of
single  center black holes, and as 
such are the states relevant for studies of the Bekenstein-Hawking area law of
black hole entropy.

Within string theory it is possible to obtain exact generating
functions of the degeneracies of classes of BPS states. The
asymptotic growth as function of the angular momentum is for example
previously studied in  \cite{Curtright:1986di, Russo:1994ev, Dabholkar:2005dt}.
String theory relates BPS states to the cohomology of moduli spaces 
of sheaves supported on a Calabi-Yau manifold. From this perspective the
$\operatorname{SU}(2)_{\rm spin}$ representations correspond to representations of
the Lefshetz $\operatorname{SL}(2)$ action on the cohomology of the
moduli space $\CM$ \cite{Witten:1996qb}. The states with $J=0$ correspond to
the part of the middle cohomology which is invariant under the Lefshetz
action. 

In the present work, we consider moduli spaces of semi-stable sheaves
supported on a complex algebraic surface $S$, which can be thought of as 
being embedded inside a Calabi-Yau manifold. If $S$ is one of the
rational surfaces, the sheaves can be related to monopole or monopole
strings in respectively four and five dimensional supersymmetric gauge theory through geometric
engineering \cite{Katz:1996fh,Morrison:1996xf}. If $S$ is a K3
surface, the sheaves correspond to (small) black holes in
$\mathcal{N}=4$ supergravity also known as Dabholkar-Harvey
states \cite{Dabholkar:1989jt, Dabholkar:2005dt}.
We specialize to the moduli space of sheaves with rank
$r=1$  and 1st and 2nd Chern classes $c_1\in H^2(S,\mathbb{Z})$ and $c_2\in H^4(S,\mathbb{Z})$. These moduli
spaces are isomorphic to the Hilbert scheme of $c_2$ points on $S$
 (viewing $c_2$ as a number). G\"ottsche has
determined the generating function of the Betti numbers of the Hilbert
schemes \cite{Gottsche:1990}. We need to introduce some notation to explain
his result. 
 
Let $\CM(n)$ be the Hilbert scheme of $n$ points. Let
furthermore 
$$P(X,\zeta):=\sum_{i=0}^{2\dim_\mathbb{C}(X)}\beta_i(X)\,\zeta^i$$
be the Poincar\'e polynomial of $X$ with $\beta_i(X)$ the $i$th Betti
number of $X$.  We choose the surface $S$ such that $\beta_1(S)=\beta_3(S)=0$. Then, we
have \cite{Gottsche:1990}:  
\[
\sum_{n\geq 0} \zeta^{-\frac{1}{2}\dim_\mathbb{C}\CM(n)}\,
P(\CM(n),\zeta^\frac{1}{2})\,q^n=q^{{\beta_2(S)+2}\over 24}\left(\zeta^{1\over
  2}-\zeta^{-{1\over 2}}\right)\,\frac{i}{\theta_1(w; \tau)\,\eta(\tau)^{\beta_2(S)-1}},
\]
which precisely equals the function in (\ref{eq:genfunc}) with
$k=\beta_2(S)+2$. The coefficients $a_{m,k}(n)$ are in this context
the Betti numbers of the moduli spaces. The expansion (\ref{eq:genfunc2}) in terms of $b_{J,k}(n)$ decomposes
the cohomology in terms of $(2J+1)$-dimensional $\operatorname{SL}(2)$ or
$\operatorname{SU}(2)_{\rm spin}$ representations. For $J=0$ and $k=24$, our formula (\ref{eq:asympalpha}) confirms
nicely the numerical estimates for $b_{0,24}(n)$ in \cite[Appendix
C]{Dabholkar:2005dt}. Note that since this analysis is carried out in
the so-called weak coupling or D-brane regime, and the coefficients
$b_{J,k}(n)$ are not BPS indices, these coefficients can not be
claimed to count  black holes with fixed angular momentum.

Equation \eqref{eq:asympalpha} shows that in the context of this
paper, the number of pure Higgs states has the same exponential growth as
the total number of states,  but that the number of pure Higgs states
is smaller by a factor $n^{-\frac{3}{2}}$. Moreover, the
number of $\operatorname{SU}(2)_{\rm spin}$ multiplets increases with $J$ for small
$J$. It is interesting to compare this with other known asymptotics of
pure Higgs states.  In particular, Ref.  \cite{Bena:2012hf, Denef:2007vg} considered this
question for quiver moduli spaces in the limit of a large number of
arrows between the nodes of a quiver with a potential. Ref.  \cite{Bena:2012hf} demonstrated that 
the number of pure Higgs states for these quivers, $\beta_{\dim_\mathbb{C}(\CM)}(\CM)-\beta_{\dim_\mathbb{C}(\CM)+1}(\CM)$, is exponentially
larger than the number of remaining $\operatorname{SL}(2)$ multiplets
given by $\beta_{\dim_\mathbb{C}(\CM)+1}(\CM)$. We note that
sheaves on toric surfaces relevant for this article also allow a
description in terms of quivers \cite{Bondal:1990, Rudakov:1990}. The Hilbert 
scheme of $n\gg 1$ points corresponds to increasing dimensions
of the spaces associated the nodes, with the number of arrows kept
fixed. Thus we observe that the asymptotic behavior of the number of
pure Higgs states in the two limits, large
number of arrows or large dimensions, is rather different. 

It will be interesting to understand better the significance of these
different asymptotic behaviors. Moreover we belief that application
of the techniques in the present paper to partition functions for higher rank sheaves on surfaces
\cite{Manschot:2011dj,Manschot:2011ym,Yoshioka:1994}, and partition 
functions of black holes and quantum geometry
\cite{deBoer:2006vg,Gaiotto:2006wm, Huang:2007sb} will lead to to
important novel insights. 

The paper is organized as follows: In Section \ref{falsetheta} we
rewrite the functions of interest in terms of false theta functions
and determine their Taylor expansion. Section \ref{circlemethod}
uses the Circle Method to prove our main theorem.

\section*{Acknowledgements}

The research of the authors was supported by the Alfried Krupp Prize for Young University Teachers of the Krupp foundation.
After this paper was submitted to the arXiv, we learned from Paul de
Lange that for special cases the main terms were also obtained by \cite{Curtright:1986di}. We are grateful to him for the correspondence. We also thank Roland Mainka, Boris Pioline, Rob Rhoades and Miguel Zapata Rol\'on for useful correspondence. 

\section{Relation to false theta functions}\label{falsetheta}
We start by writing the generating function of $a_{m,k}(n)$ for fixed
$m$ and $k$ in terms of the functions $\vartheta_m(q)$
defined by:
\[
\vartheta_m \left(q\right) := \left(1+q^{|m|}\right)\sum_{n\geq 0}(-1)^n q^{\frac{n(n+1)}{2}+n|m|}-1.
\] 
\begin{remark}
We note that the property $\vartheta_m \left(q\right)=\vartheta_{-m} \left(q\right)$ of $\vartheta_m \left(q\right)$, continues to hold when $\vartheta_m \left(q\right)$ is defined with $|m|$ replaced by $m$. 
\end{remark}
The functions $\vartheta_m \left(q\right)$ are examples of false theta functions, which were
first introduced by Rogers \cite{Rogers:1917} and have attracted a lot of interest recently. Using the Rogers and Fine identity one can relate $\vartheta_m$ to so-called quantum modular forms which are functions mimicking modular behavior on (subsets of) $\Q$ \cite{FOR}.

It is well-known that the inverse theta function $\theta_1(w;\tau)^{-1}$ can be written as a
sum over its poles. See for more details for example \cite{AG, Garvan:1988}.

\begin{proposition}\label{qkeqn}
We have for $|\zeta q|,\,|\zeta^{-1} q|<1$: 
$$
q^{\frac{k}{24}}\left(\zeta^{1\over 2}-\zeta^{-{1\over
        2}}\right)h_{k}(w;\tau)=\frac{1}{(q)_\infty^k} \sum_{m\in \mathbb{Z}}\vartheta_m(q) \zeta^m
$$
with $(q)_\infty:=\prod_{n=1}^\infty (1-q^n)$.
\end{proposition}

\begin{proof}
The inverse theta function $\theta_1(w;\tau)^{-1}$ is expressed as a
sum over its poles by:
\begin{equation}\label{findcrank}
 \frac{iq^{\frac18}\left( \zeta^{\frac12} - \zeta^{-\frac12}
  \right)}{\theta_1 \left( w; \tau\right)} =
\frac1{(q)_\infty^3}\left( 1- \zeta\right) \sum_{n\in\Z} \frac{(-1)^n
  q^{\frac{n(n+1)}2}}{1-\zeta q^n}.
\end{equation}
Using geometric series expansion, we may rewrite (\ref{findcrank}) as
\[
\frac{1}{(q)_\infty^3}+\frac{1}{(q)_\infty^3} (1-\zeta)\sum\limits_{n> 0\atop{m\geq 0}}(-1)^n q^{\frac{n(n+1)}{2}+nm}\zeta^m +\frac{1}{(q)_\infty^3} \left(1-\zeta^{-1}\right)
\sum\limits_{n>0\atop{m\geq 0}}(-1)^n q^{\frac{n(n+1)}{2}+nm}\zeta^{-m}.
\]  
From this the statement of the proposition easily follows.
\end{proof} 
The function $\vartheta_m$ is not modular but may be nicely approximated by its Taylor expansion. For this we use the following general lemma (see \cite{Zagier:2006} for the case of real functions).
\begin{lemma}\label{ZagierTaylor}
Let $f:\C\rightarrow\C$ be a $C^\infty$
function. Furthermore, we require that $f(x)$ and all its
derivatives are of rapid decay for $\re(x) \to \infty$.
Then for $t\to 0$ with $\re(t)>0$ and
$a>0$, we have for any $N\in \N_0$: 
$$
\sum_{m=0}^\infty f\left( (m+a) t\right) = \frac1{t} \int_0^\infty f(x) dx - \sum_{n=0}^N \frac{f^{(n)} (0)}{n!} \frac{B_{n+1} (a)}{n+1} t^n+O\left(t^{N+1}\right).
$$
Here $B_n(x)$ denotes the $n$th Bernoulli polynomial. 
\end{lemma}

To use Lemma \ref{ZagierTaylor} we write for fixed $N\geq 1$ and $q=e^{-z}$ 
\begin{equation}
\label{eq:ZagierTaylor}
q^{\frac{k}{24}} \vartheta_m(q) =: \sum_{\ell=1}^N d_{m,k}(\ell) z^\ell
+O\left(z^{N+1}\right). 
\end{equation}
Lemma \ref{ZagierTaylor} then gives 

\begin{lemma}\label{ThetaLemma}
We have for $N\geq 1$
$$
\vartheta_m \left( q\right) = \left( 1 +q^m\right) q^{-\frac12 \left(m+\frac12\right)^2} \sum_{\ell=0}^N c_m (\ell) z^\ell-1 + O\left( z^{N+1}\right)
$$
with
$$
c_m (\ell) := \frac{(-1)^{\ell-1} 2^{\ell}}{\ell! (2\ell+1)} \left( B_{2\ell+1} \left( \frac{m}{2}+\frac14 \right) - B_{2\ell +1} \left( \frac{m}{2}+\frac34\right)\right).
$$
 In particular the first values for $d_{m,k}(\ell)$ are:
\[
d_{m,k}(1)=\frac14, \qquad d_{m,k}(2)= -\frac{k}{96} + \frac{1}{16}, \qquad  d_{m,k}(3)=-\frac{m^2}{16} -\frac{k}{384} + \frac{k^2}{4608} + \frac{5}{192} . 
\]
\end{lemma}

\begin{proof}
We may write $\vartheta_m(q)$ as:
\begin{eqnarray*}
\vartheta_m(q)&=&\left(1+q^m\right)
q^{-\frac12\left(m+\frac12\right)^2} \\
&& \times \sum_{n\geq 0}\left(f\left(\left(n+\frac{m}{2}+\frac14\right)\sqrt{z}\right)-f\left(\left(n+\frac{m}{2}+\frac34\right)\sqrt{z}\right)\right)-1
\end{eqnarray*}
with $f(x):=e^{-2x^2}$.  Substitution of Lemma \ref{ZagierTaylor} gives the
desired result.
\end{proof}
\begin{remark}
We note that the case $m=0$ can be easily concluded from \cite{Wright:1971}, where the asymptotics of the coefficients of $1/2(1-\vartheta_0(q))/(q)_\infty^k$ for $k=1,2$ are determined.
\end{remark}

\section{Use of the Circle Method}\label{circlemethod}

In this section, we prove Theorem \ref{mainasymp} following an approach by Wright \cite{Wright:1971}.
To prove the theorem, we assume via symmetry that $m\geq 0$ 
and set
$$
F_{m,k}(q):=\sum_{n \geq 0}a_{m,k}(n)q^n.
$$
By Proposition \ref{qkeqn}, we obtain that 
$$
F_{m,k}(q) = \frac{1}{(q)_{\infty}^k}\vartheta_{m}(q).
$$
 By Cauchy's Theorem, we have for $n\geq 1$
$$
a_{m,k} (n) = \frac1{2\pi i} \int_{\mathcal{C}} \frac{F_{m,k} (q)}{q^{n+1}} dq,
$$
where $\mathcal{C}$ is a circle surrounding $0$ counterclockwise.  We choose  $e^{-\eta}$ for the radius of $\mathcal{C}$ with $\eta=\pi\sqrt{k \over 6n}$ and split $\mathcal{C}$ into two arcs
$\mathcal{C}=\mathcal{C}_1+\mathcal{C}_2$,
where $\mathcal{C}_1$ is the arc going counterclockwise from phase
$-2\eta$ to $2\eta$ and $\mathcal{C}_2$ is its
complement in $\mathcal{C}$. Consequently, we have 
\[
a_{m,k} (n)=M+E
\]
with
\begin{align*}
M&:=\frac{1}{2\pi i}\int_{\mathcal{C}_1} \frac{F_{m,k}(q)}{q^{n+1}}dq,\\
E&:=\frac{1}{2\pi i}\int_{\mathcal{C}_2} \frac{F_{m,k}(q)}{q^{n+1}}dq.
\end{align*}
We will show that the main asymptotic contribution comes from $M$. Moreover we parametrize $q=e^{-z}$ with $\re (z) = \eta$.
  
\subsection{The integral along $\mathcal{C}_1$}\label{sec:intC1}
In the integral along $\mathcal{C}_1$,  we approximate $F_{m,k}$ by
simpler functions. Firstly, recall that from the transformation law of
the $\eta$-function \cite[Theorem 3.1]{Apostol:1976}  we obtain: 
\begin{equation}\label{eta-trans}
\frac1{\left(e^{-z}; e^{-z} \right)_\infty} = \sqrt{\frac{z}{2\pi}} e^{-\frac{z}{24} + \frac{\pi^2}{6z}}\frac{1}{ \left( e^{-\frac{4\pi^2}{z}}; e^{-\frac{4\pi^2}{z}} \right)_\infty}.
\end{equation}
Thus we want to approximate $\frac1{(q)_\infty^k}$ by
$$
z^{\frac{k}2} (2\pi)^{-\frac{k}{2}} e^{- \frac{kz}{24}} P_k \left( e^{-\frac{4\pi^2}{z}} \right),
$$
where
$$
P_k (q) := \left( 1 + \sum_{\substack{24\ell - k <0 \\ \ell>0 }} p_k (\ell) q^\ell \right) q^{- \frac{k}{24}}.
$$
To be more precise, we split
$$
M=M_1+E_1
$$  
with
$$
M_1 := \frac1{2\pi i} \int_{\mathcal{C}_1} \frac1{q^{n+1}} \left( \frac{z}{2\pi}\right)^{\frac{k}2} e^{-\frac{kz}{24}} P_k \left(e^{- \frac{4\pi^2}{z}} \right) \vartheta_m (q) dq,
$$
$$
E_1 := \frac1{2\pi i} \int_{\mathcal{C}_1} \frac1{q^{n+1}} \vartheta_m (q) \left( \frac1{\left( e^{-z} ; e^{-z}\right)_\infty^k} - \left( \frac{z}{2\pi}\right)^{\frac{k}2} e^{-\frac{kz}{24}} P_k \left(e^{- \frac{4\pi^2}{z}} \right)\right) dq.
$$
We first bound $E_1$ which turns into the error term. Firstly we obtain from (\ref{eta-trans})
\[
\frac1{\left( e^{-z} ; e^{-z}\right)_\infty^k} - \left( \frac{z}{2\pi}\right)^{\frac{k}2} e^{-\frac{kz}{24}} P_k \left(e^{- \frac{4\pi^2}{z}} \right) =
O(1).
\]
To bound $\vartheta_m$, we use that on $\mathcal{C}_1$
$$
|z|^2 = \eta^2 +\im (z)^2 \leq \eta^2 +4\eta^2.
$$
Thus, by Lemma \ref{ThetaLemma},
$$
\left| \vartheta_m \left(q\right) \right| \ll |z|\ll\eta,
$$
where  throughout $g(x)\ll f(x)$ has the same meaning as $g(x)=O(f(x))$.
Using that the length of $\mathcal{C}_1$ is $O(\eta)$, we may thus bound
$$
E_1
 \ll n^{-1}   e^{\pi \frac{\sqrt{kn}}{\sqrt{6}}}.
$$

We next investigate $M_1$.
We aim to approximate $\vartheta_m$ by its Taylor expansion given in Lemma \ref{ThetaLemma} and thus we split
$$
M_1=M_2+E_2
$$
with 
$$
M_2:= \frac1{2\pi i}\sum_{\ell=1}^N d_{m,k}(\ell) \int_{\mathcal{C}_1} \frac1{q^{n+1}} \left( \frac{z}{2\pi}\right)^{\frac{k}2} P_k \left(e^{- \frac{4\pi^2}{z}} \right) z^\ell dq,
$$
$$
E_2:= \frac1{2\pi i} \int_{\mathcal{C}_1} \frac1{q^{n+1}} \left(\frac{z}{2\pi}\right)^{\frac{k}2} P_k \left(e^{- \frac{4\pi^2}{z}} \right) \left(e^{-\frac{kz}{24}}\vartheta_m(q)- \sum_{\ell=1}^N d_{m,k}(\ell) z^\ell \right)dq.
$$
We first estimate $E_2$ and show that it contributes to the error term.  By Lemma \ref{ThetaLemma}
$$
E_2  \ll \int_{\mathcal{C}_1} e^{n\eta} \left|z\right|^{\frac{k}2 +N+1} e^{\frac{\pi^2 k}6 \re \left( \frac1{z}\right)}dz.
$$
Since
$$
\frac{\re(z)}{|z|^2} \leq \frac{1}{\re(z)} = \frac1{\eta}
$$
we may  bound
$$
n\eta +\frac{\pi^2 k}6 \re \left( \frac{1}{z}\right) \leq \pi \sqrt{\frac{2kn}{3}}
$$
Moreover on $\mathcal{C}_1$
$$
|z|^2 = x^2+y^2\leq \eta^2+4\eta^2 \ll \eta^2.
$$
As before, the path of integration may be estimated against $\eta$. Thus
$$
\left|E_2 \right| \ll \eta^{N+\frac{k}2 +2}e^{\pi\sqrt{\frac{2kn}3}}\ll n^{-\frac{N}{2}-\frac{k}4 -1} e^{\pi\sqrt{\frac{2kn}3}}.
$$
We next decompose 
\be
\label{eq:I12}
M_2=(2\pi)^{-\frac{k}2} \sum_{\ell=1}^Nd_{m,k}(\ell) \sum_{24j-k\leq 0} p_k (j) \mathcal{I}_{\ell + \frac{k}2, j}^{(k)},
\ee
where for $s>0$
\[
\mathcal{I}_{s,j}^{(k)} :=\frac{1}{2\pi i}\int_{\mathcal{C}_1}z^s e^{\frac{\pi^2 k}{6z} - \frac{4\pi^2 j}{z} +(n+1)z}dq.
\] 
These integrals may now be written in terms of the classical $I$-Bessel functions.
\begin{lemma}\label{lemma:Isexp} We have
$$
\mathcal{I}_{s,j}^{(k)} = n^{\frac{-s-1}{2}}\left( \frac{\pi\sqrt{k-24j}}{\sqrt{6}} \right)^{s+1} I_{-s-1}\left(\pi \sqrt{\frac23 (k-24j)\,n} \right)+O\left(n^{-1-\frac{s}{2}}e^{\frac{\pi\sqrt{3kn}}{2\sqrt{2}}}\right).
$$ 
\end{lemma}
\begin{proof}
Let $\mathcal{D}$ be the rectangular counterclockwise
path from $-\infty-2i\eta$ to $-\infty+2i\eta$  with endpoints 
$\eta-2i\eta$ and $\eta+2i\eta$.  
Denote by $\mathcal{D}_{i}$, $i=1,2,3$,
the paths 1) from  $-\infty-2i\eta$ to $\eta-2i\eta$, 2) from $\eta-2i\eta$ to
$\eta+2i\eta$, and 3) from $\eta+2i\eta$ to $-\infty+2i\eta$. Making the change of variables $q=e^{-z}$ gives that
$$
\mathcal{I}_{s,j}^{(k)}=\frac{1}{2\pi i}\int_{\mathcal{D}_2} z^s e^{\frac{\pi^2 k}{6z} - \frac{4\pi^2 j}{z}+nz} dz.
$$
We next use the Residue Theorem to turn this integral into an integral over $\mathcal{D}$. For this we bound the integrals along $\mathcal{D}_1$ and $\mathcal{D}_3$. We
only give the details for $\mathcal{D}_3$. On this path we may bound
$$
\left| \re \left( \frac1{z} \right) \right|  \leq \frac1{|z|} \leq \frac1{2\eta}.
$$
Writing $z=\eta (1+2i) -u$, $0\leq u<\infty$ gives
$$
|z| = \sqrt{\left( \eta -u \right)^2 +4\eta^2} \ll \eta +u.
$$
Thus the integral along $\mathcal{D}_3$ may be bounded by
\begin{equation}\label{boundD3}
\ll e^{\frac{2 \pi^2}{\eta}\left( \frac{k}{24} - j \right)}
\int_0^\infty \left( \eta +u\right)^s e^{n (\eta -u)} du \ll
e^{\frac{\pi^2 k}{12\eta} +n\eta} \left( \eta^s \int_0^\eta e^{-nu}du +
  \int_\eta^\infty u^s e^{-nu} du\right). 
\end{equation}
The second term is an incomplete Gamma function and thus exponentially small. Thus (\ref{boundD3}) may up to an exponentially small error be bounded by
$$
\eta^s e^{\frac{\pi^2 k}{12\eta}+n\eta}\frac1n \left(1-e^{-n\eta}\right)\ll n^{-1-\frac{s}{2}}e^{\frac{\pi\sqrt{3kn}}{2\sqrt{2}}}.
$$
In the remaining integral we make the change of variables $z=\frac{t}{n}$ to get
\be
\label{eq:Isj}
\mathcal{I}_{s,j}^{(k)}=n^{-s-1} \frac1{2\pi i} \int_\mathcal{D} t^s e^{t+\frac{\pi^2 kn}{6t}-\frac{4\pi^2 jn}{t}} dt+O\left(n^{-1-\frac{s}{2}}e^{\frac{\pi \sqrt{3kn}}{2\sqrt{2}}}\right).
\ee
We now use the following representation of the $I$-Bessel function
\cite{Abramovitz:1972} 
\[
I_\ell (2\sqrt{z}) =  z^{\frac{\ell}2} \frac1{2\pi i} \int_{-\infty}^{(0+)} t^{-\ell-1} \text{exp} \left( t+\frac{z}{t} \right) dt,
\]
where the integral is along any path looping from $-\infty$ around $0$
back to $-\infty$ counterclockwise. Substitution into (\ref{eq:Isj}) gives the claim.
\end{proof} 

Substitution of Lemma \ref{lemma:Isexp} in equation (\ref{eq:I12}) yields
\begin{align*}
M_2&=(2\pi)^{-\frac{k}2}\sum_{\substack{1\leq \ell \leq N \\ 24j-k \leq 0}} d_{m,k}(\ell)\,p_k(j)\, n^{-\frac{2+2\ell+k}{4}}\left(\frac{\pi\sqrt{k-24j}}{\sqrt{6}}\right)^{1+\ell +\frac{k}{2}}
\\
&\qquad \qquad \times I_{-1-\ell-\frac{k}2}\left(\pi\sqrt{\frac23 (k-24j)\,n} \right)+O\left(n^{-\frac{3}{2}-\frac{k}{4}} e^{\frac{\pi \sqrt{3kn}}{2\sqrt{2}}}\right).
\end{align*}

\subsection{The integral along $\mathcal{C}_2$}
On $\mathcal{C}_2$, $\im(z)$ varies from $-2\eta$ to $-2\pi + 2\eta$. 
Using a rough bound for the theta function, we find
\[
\left\lvert \vartheta_m(q)\right\rvert \ll 2\sum_{n\geq 1} e^{-\frac{n}{2}(n+1+2m)\text{Re}(z)}  +1
\ll 2\sum_{n\geq 0} e^{-n\text{Re}(z)}=\frac{2}{1-e^{-\eta}}\ll\frac1\eta.
\]
Using (\ref{eta-trans}) we obtain the bound
\begin{equation}\label{etabound}
\frac{1}{\left( e^{-z} ; e^{-z}\right)_\infty} \ll e^{\frac{\pi^2}{6}\text{Re}\left(\frac1z\right)}.
\end{equation}
Now
\[
\text{Re}\left(\frac1z\right)=\frac{\eta}{\eta^2+\im(z)^2}\leq \frac{1}{5\eta}.
\]
Thus (\ref{etabound}) may be estimated against $\exp({\frac{\pi^2}{30\eta}})$.
This gives that:
\[
E \ll e^{\frac{\pi^2 k}{30\eta}+n\eta}\ll e^{\frac{\pi}5 \sqrt{6nk}}.
\]
This is exponentially smaller than the other errors. Combining this with the results of Subsection \ref{sec:intC1} therefore gives Theorem \ref{mainasymp}.

\providecommand{\href}[2]{#2}\begingroup\raggedright


\begin{thebibliography}{99}

\bibitem{Abramovitz:1972}
M.~Abramovitz and I.~Stegun, {\it Handbook of Mathematical Functions
  with Formulas, Graphs, and Mathematical Tables},  New York: Dover
  Publications (1972). 

\cite{AG}
\bibitem{AG}
G.~Andrews and F.~Garvan, {\it Dyson's crank of a partition}, Bull. Amer. Math. Soc. {\bf 18} (1988), 161-171. 

\bibitem{Apostol:1976}
T. Apostol, {\it Modular Functions and Dirichlet Series in Number
  Theory}, Springer-Verlag (1976). 

\bibitem{Bena:2012hf}
  I.~Bena, M.~Berkooz, J.~de Boer, S.~El-Showk, and D.~van den Bleeken,
  {\it Scaling BPS Solutions and pure-Higgs States,}
  JHEP {\bf 1211} (2012) 171
  [arXiv:1205.5023 [hep-th]].

\bibitem{deBoer:2006vg}
  J.~de Boer, M.~Cheng, R.~Dijkgraaf, J.~Manschot, and E.~Verlinde,
  {\it A Farey Tail for Attractor Black Holes,}
  JHEP {\bf 0611} (2006) 024
  [hep-th/0608059].

\bibitem{Bondal:1990}
A.~Bondal, {\it Representation of  associative algebras and coherent sheaves}, Math. USSR Izvestiya {\bf 34} (1990) 23-42. 

\bibitem{Curtright:1986di}
  T.~L.~Curtright and C.~B.~Thorn,
{\it Symmetry Patterns In The Mass Spectra Of Dual String Models,}
  Nucl.\ Phys.\ B {\bf 274} (1986) 520.

\bibitem{Dabholkar:1989jt}
  A.~Dabholkar and J.~A.~Harvey,
  {\it Nonrenormalization of the Superstring Tension,}
  Phys.\ Rev.\ Lett.\  {\bf 63} (1989) 478.

\bibitem{Dabholkar:2005dt}
  A.~Dabholkar, F.~Denef, G.~W.~Moore and B.~Pioline,
  {\it Precision counting of small black holes,}
  JHEP {\bf 0510} (2005) 096
  [hep-th/0507014].

\bibitem{Dabholkar:2012nd}
  A.~Dabholkar, S.~Murthy, and D.~Zagier,
  {\it Quantum Black Holes, Wall Crossing, and Mock Modular Forms,}
  arXiv:1208.4074 [hep-th].

\bibitem{Denef:2007vg}
  F.~Denef and G.~Moore,
  {\it Split states, entropy enigmas, holes and halos,}
  JHEP {\bf 1111} (2011) 129
  [hep-th/0702146 [HEP-TH]].

\bibitem{DiFrancesco:1997nk}
  P.~Di Francesco, P.~Mathieu, and D.~Senechal,
 {\it Conformal field theory,}
 New York, USA: Springer (1997).

\bibitem{Dijkgraaf:1996xw}
  R.~Dijkgraaf, G.~Moore, E.~Verlinde, and H.~Verlinde,
  {\it Elliptic genera of symmetric products and second quantized strings,}
  Commun.\ Math.\ Phys.\  {\bf 185} (1997) 197-209
  [hep-th/9608096].

\bibitem{eichler}
M.~Eichler and D.~Zagier, {\em The Theory of Jacobi Forms}.
\newblock Birkh\"auser, 1985.

\bibitem{FOR}
  A.~Folsom, K.~Ono, and R.~Rhoades,
  {\it $q$-series and quantum modular forms,}
  submitted for publication.


\bibitem{Gaiotto:2006wm}
  D.~Gaiotto, A.~Strominger, and X.~Yin,
  {\it The M5-Brane Elliptic Genus: Modularity and BPS States,}
  JHEP {\bf 0708} (2007) 070
  [hep-th/0607010].


\bibitem{Garvan:1988}
F.~G.~Garvan, {\it New Combinatorial Interpretations of Ramanujan's Partition
  Congruences Mod 5,7 and 11}, Trans. Amer. Math. Soc. {\bf 305}
(1988) 47 -77.

\bibitem{Gottsche:1990}
L.~G\"ottsche, {\it The Betti numbers of the Hilbert scheme of points
  on a smooth projective surface}, Math.\ Ann. {\bf 286} (1990), 193--207.

\bibitem{Gottsche:1996}
  L.~G\"ottsche and D.~Zagier,
  {\it Jacobi forms and the structure of Donaldson invariants for 4-manifolds with $b_+=1$,}
Selecta \ Math. New Ser. {\bf 4} (1998) 69-115.   
  [arXiv:alg-geom/9612020].

\bibitem{Hammond:2004}
P.~Hammond and R.~Lewis, {\it Congruences in ordered pairs of
  partitions}, Int. J. Math. Sci.  2004, nos. 45-48, 2509-2512.
 

\bibitem{HR2} G.  Hardy and S. Ramanujan,
\emph{Asymptotic formulae in combinatory analysis,}
Collected papers of Srinivasa Ramanujan \textbf{244}, AMS Chelsea Publ.,
Providence, RI, 2000.

\bibitem{Huang:2007sb}
  M.~-x.~Huang, A.~Klemm, M.~Marino and A.~Tavanfar,
  {\it Black holes and large order quantum geometry,}
  Phys.\ Rev.\ D {\bf 79} (2009) 066001
  [arXiv:0704.2440 [hep-th]].

\bibitem{Katz:1996fh}
  S.~Katz, A.~Klemm, and C.~Vafa,
 {\it Geometric engineering of quantum field theories,}
  Nucl.\ Phys.\ B {\bf 497} (1997) 173-195
  [hep-th/9609239].

\bibitem{Kawai:1993jk}
  T.~Kawai, Y.~Yamada, and S.~Yang,
  {\it Elliptic genera and N=2 superconformal field theory,}
  Nucl.\ Phys.\ B {\bf 414} (1994) 191-212
  [hep-th/9306096].

\bibitem{Lee:2012sc}
  S.~Lee, Z.~Wang, and P.~Yi,
{\it Quiver Invariants from Intrinsic Higgs States,}
  JHEP {\bf 1207} (2012) 169
  [arXiv:1205.6511 [hep-th]].
 
\bibitem{Manschot:2011dj}
  J.~Manschot,
  {\it BPS invariants of N=4 gauge theory on a surface,}
  Commun.\  Num.\  Theor.\  Phys.\  {\bf 6} (2012) 497-516
  [arXiv:1103.0012 [math-ph]].

\bibitem{Manschot:2011ym}
  J.~Manschot,
 {\it BPS invariants of semi-stable sheaves on rational surfaces,}
 Lett.\ Math.\ Phys.\ {\it to appear},
  arXiv:1109.4861 [math-ph].

\bibitem{Manschot:2012rx}
  J.~Manschot, B.~Pioline, and A.~Sen,
  {\it From Black Holes to Quivers,}
  JHEP {\bf 1211} (2012) 023
  [arXiv:1207.2230 [hep-th]].

\bibitem{Morrison:1996xf}
  D.~Morrison and N.~Seiberg,
  {\it Extremal transitions and five-dimensional supersymmetric field theories,}
  Nucl.\ Phys.\ B {\bf 483} (1997) 229-247
  [hep-th/9609070].

\bibitem{rademacher1938:1}
H.~Rademacher and H.~Zuckerman, {\it On the fourier coefficients of certain
  modular forms of positive dimension},  Annals of Mathematics {\bf 39}
  (1938), no.~2 433--462.

\bibitem{Rogers:1917}
L.~Rogers, {\it On two theorems of combinatory analysis and some allied identities,} Proc. London Math. Soci. (2),
{\bf 16} (1917), 315-336.

\bibitem{Rudakov:1990}
A.~N.~Rudakov, {\it Helices and vector bundles: seminaire Rudakov}, Cambridge University Press (1990)

\bibitem{Russo:1994ev}
  J.~G.~Russo and L.~Susskind,
  {\it Asymptotic level density in heterotic string theory and rotating black holes,}
  Nucl.\ Phys.\ B {\bf 437} (1995) 611
  [hep-th/9405117].


\bibitem{Sen:2009vz}
  A.~Sen,
 {\it Arithmetic of Quantum Entropy Function,}
  JHEP {\bf 0908} (2009) 068
  [arXiv:0903.1477 [hep-th]].

\bibitem{Yoshioka:1994}
K.~Yoshioka, {\it The Betti numbers of the moduli space of stable sheaves of
  rank 2 on $\mathbb{P}^2$},  J. reine. angew. Math. {\bf 453} (1994)
  193--220.

\bibitem{Witten:1996qb}
  E.~Witten, 
  {\it Phase transitions in M theory and F theory,}
  Nucl.\ Phys.\ B {\bf 471} (1996) 195-216
  [hep-th/9603150].

\bibitem{Wright:1968}
E.~Wright, {\it Stacks }, Q. J. Math. {\bf 19}  (1968), 313-320.


\bibitem{Wright:1971}
E.~Wright, {\it Stacks (II)}, Q. J. Math. {\bf 22}  (1971), 107-116.

\bibitem{Zagier:2006}
D.~Zagier, {\it The Mellin transform and other useful analytic
  techniques}, Appendix to E. Zeidler, {\it Quantum Field Theory I:
  Basics in Mathematics and Physics. A Bridge Between Mathematicians
  and Physicists}, Springer-Verlag (2006) 305-323.

\bibitem{Z2}
D.~Zagier, {\it Quantum modular forms}, In Quanta of Maths: Conference in honour 
of Alain Connes, Clay Mathematics Proceedings {\bf 11}, AMS and Clay Mathematics Institute 2010, 659-675.


\end{thebibliography}
\end{document}